\documentclass[a4paper,12pt,twoside,openright]{amsart}
\usepackage[english]{babel}
\usepackage[utf8]{inputenc}
\usepackage{fancyhdr}

\usepackage{amssymb}
\usepackage{amsmath} 
\usepackage{amsthm} 
\usepackage{enumitem}
\usepackage{array}
\usepackage{tikz-cd}
\usepackage{extarrows}
\usepackage{mathtools}
\usepackage{faktor}
\usepackage{float}

\makeatletter
\def\subsubsection{\@startsection{subsubsection}{3}%
  \z@{.5\linespacing\@plus.7\linespacing}{-.5em}%
  {\normalfont\bfseries}}
\makeatother

\theoremstyle{plain} 
\newtheorem{thm}{Theorem}[section] 
\newtheorem{mthm}{Main Theorem}[]
\newtheorem{cor}[thm]{Corollary} 
\newtheorem{lem}[thm]{Lemma} 
\newtheorem{prop}[thm]{Proposition}

\theoremstyle{definition} 
\newtheorem{defn}{Definition}[] 
\newtheorem{es}{Example}

\theoremstyle{remark} 
\newtheorem{oss}{Remark}

\makeatletter 
\newcommand{\xhooklongrightarrow}[2][]{%
  \ext@arrow3399{\hooklongrightarrowfill@}{#1}{#2}} 
\makeatother

 \DeclareMathOperator{\Ass}{Ass}

 \DeclareMathOperator{\hgt}{ht}

  \DeclareMathOperator{\car}{char}
    \DeclareMathOperator{\Tr}{Tr}

   \DeclareMathOperator{\lt}{in_\prec}
    
   \DeclareMathOperator{\init}{in}
 \DeclareMathOperator{\Min}{Min}

 \usepackage{blindtext}
\title{Knutson ideals and determinantal ideals of Hankel matrices}
\author{Lisa Seccia}
\address{Universit\`a  di Genova,  Dipartimento di Matematica. 
 Via Dodecaneso 35, 16146 Genova, Italy}
\email{seccia@dima.unige.it}
 \date{}
 
\begin{document}
 \maketitle
 \begin{abstract}
 Motivated by a work of Knutson, in a recent paper Conca and Varbaro have defined a new class of ideals, namely “Knutson ideals”, starting from a polynomial $f$ with squarefree leading term. We will show that the main properties that this class has in polynomial rings over fields of characteristic $p$ are preserved when one introduces the definition of Knutson ideal also in polynomial rings over fields of characteristic zero. Then we will show that determinantal ideals of Hankel matrices are Knutson ideals for a suitable choice of the polynomial $f$.
 \end{abstract}
 \section{Introduction}
 Let $\mathbb{K}$ be a field of any characteristic. Fix $f \in S= \mathbb{K}[x_1,\ldots,x_n]$ a polynomial such that its leading term $\lt (f)$ is a squarefree monomial  for some term order $\prec$. We can define a new family of ideals starting from the principal ideal $(f)$ and taking associated primes, intersections and sums. Geometrically  this means that we start from the hypersurface defined by $f$ and we construct a family of new subvarieties $\{Y_i\}_i$ by taking irreducible components, intersections and unions.\par
In \cite{CV}, Conca and Varbaro called this class of ideals \emph{Knutson ideals}, since they were first studied by Knutson in \cite{Kn}.
 \begin{defn}[\textbf{Knutson ideals}] \label{K.I.} Let $f \in S= \mathbb{K}[x_1,\ldots,x_n]$ be a polynomial such that its leading term $\lt (f)$ is a squarefree monomial  for some term order $\prec$ .
Define $\mathcal{C}_f$ to be the smallest set of ideals satisfying the following conditions:
\begin{enumerate}
\item[1.] $(f) \in \mathcal{C}_f$;
\item[2.]  If $I \in \mathcal{C}_f$ then $I:J \in \mathcal{C}_f$ for every ideal $J \subseteq S$;
\item[3.] If $I$ and $J$ are in $\mathcal{C}_f$ then also $I+J$ and $I \cap J$ must be in $\mathcal{C}_f$.
\end{enumerate} 
If $I$ is an ideal in $\mathcal{C}_f$, we say that $I$ is a \emph{Knutson ideal associated to} $f$. More generally, we say that $I$ is a \emph{Knutson ideal} if $I \in \mathcal{C}_f$ for some $f$.
 \end{defn}
 
 In \cite{Kn} Knutson proved that if $\mathbb{K}=\mathbb{Z}/p\mathbb{Z}$,  this class of ideals has some interesting properties.
 
 \begin{thm} \label{kn-fp}
Let $S=\mathbb{Z}/p\mathbb{Z}[x_1, \ldots,x_n]$ and let $f$ be a polynomial in $S$ such that $\init_{\prec}(f)= \prod_{i} x_i$ with respect to some term order.
\begin{enumerate}[label=(\roman*)]
\item \cite[Theorem 4]{Kn} If $I \in \mathcal{C}_f$, then $\init \left(I\right)$ is squarefree. In particular, $I$ is radical. \label{kn-fp1}
\item \cite[Corollary 2]{Kn} If $I, J \in \mathcal{C}_f$ then $\mathcal{G}_{I+J}=\mathcal{G}_I \cup \mathcal{G}_J$
where $\mathcal{G}_I$ (respectively $\mathcal{G}_J$) is a Gr\"obner basis of $I$ (respectively $J$). 
\end{enumerate} 
\end{thm}

\begin{oss}\label{finite}
Note that if $C$ is a family of ideals closed under intersections and such that $\lt (I)$ is squarefree for every $I \in C$, then $C$ is a finite set. In fact, it is easy to check that
$$\lt (I \cap J) \subseteq \lt (I) \cap \lt (J) \subseteq \sqrt{\lt(I \cap J)}.$$
Since $C$ is closed under intersections, $I\cap J \in C$ and therefore $ \lt (I \cap J)=\sqrt{\lt(I \cap J)}$. So, from the previous chain of subsets, we get $$\lt (I \cap J) = \lt (I) \cap \lt (J). $$ More generally, this holds for every finite intersection: $$\lt( \bigcap \limits_i I_i)=\bigcap \limits_i \lt (I_i).$$ We claim that if $I,J \in C$ and $I \neq J$, then $\lt (I) \neq \lt(J)$ and since $\lt (I)$ is squarefree for every $I \in C$, these initial ideals are a finite number. Hence $C$ is finite.
To prove the claim, assume that $\lt (I)=\lt(J)$. Then $$\lt(I\cap J)=\lt (I) \cap\lt(J)=\lt(I)=\lt(J).$$
Considering that $I \cap J \subseteq I,J$, we get $I=I\cap J=J$. This completes the proof of the claim.
\end{oss}

From Theorem \ref{kn-fp}.\ref{kn-fp1} and Remark \ref{finite}, one can infer that $\mathcal{C}_f$ is finite.\\


\begin{oss}
 Actually, assuming that every ideal of $\mathcal{C}_f$ is radical, the second condition in Definition \ref{K.I.} can be replaced by the following:
 \begin{itemize}
 \item[$2^\prime .$]If $I \in \mathcal{C}_f$ then $\mathcal{P} \in \mathcal{C}_f$ for every $\mathcal{P} \in \Min(I)$.
 \end{itemize}
 In fact, let $I\in \mathcal{C}_f$, then
 $$I= \sqrt{I}=P_1 \cap P_2 \cap \ldots\cap P_r$$
 where $P_i$ are the minimal primes of $I$. Fix $c \in (P_2 \cap \ldots\cap P_r) \setminus P_1$. Clearly $P_1 \subseteq (I:c) \subseteq P_1$, hence $P_1= (I:c)$. The same holds for every $P_i$. Viceversa, it is easy to observe that if $I$ is radical, then the minimal primes of $I:J$ are exactly the minimal primes of $I$ that do not contain $J$.
 
 \end{oss}

In this paper we begin the study of this class of ideals whose properties allow us to prove interesting results on radicality and $F$-purity of certain ideals.\par
In Section 2, we introduce the definition of Knutson ideals in polynomial rings over any field and we show that the properties listed in the previous discussion stay unchanged. In particular, we start by proving the following:
\begin{mthm} Let $S=\mathbb{K}[x_1, \ldots,x_n]$ be a  polynomial ring over any field and let $f$ be a polynomial in $S$ such that $\init_{\prec}(f)= \prod_{i} x_i$ with respect to some term order. If $I \in \mathcal{C}_f$, then $\init_{\prec}\left(I\right)$ is squarefree. In particular, $I$ is radical.
\end{mthm}
 To do so, we first generalize Knutson's result \cite[Theorem 4] {Kn} to fields of positive characteristic (see Proposition \ref{pp1}) and then we use the achieved result together with reduction modulo $p$ to prove that the same holds for polynomial rings over fields of characteristic $0$ (see Proposition \ref{car0}). Once we have proved these results, the finitness of the family $\mathcal{C}_f$ can be inferred again from Remark \ref{finite}, while the last property about Gr\"obner bases can be deduced  by Remark \ref{finite} using the fact that $$\lt(I\cap J)=\lt(I) \cap \lt(J) \Leftrightarrow \lt(I+ J)=\lt(I) + \lt(J).$$ 
In the case of homogeneous ideals, the latter equivalence comes from the usual short exact sequence
$$0 \longrightarrow S/(I \cap J) \longrightarrow S/I \oplus S/J \longrightarrow S/(I+J)\longrightarrow 0$$
using the fact that the Hilbert function does not change when passing to the inital ideal. If $I$ and $J$ are not homogeneous, the equivalence is still true but the proof requires more work.\par
In Section 3, we discuss the case of determinantal ideals of generic Hankel matrices and we prove that they are Knutson ideals for a suitable choice of $f$ (see Theorem \ref{prophank} and Theorem \ref{prophank2}):
\begin{mthm} Let $H$ be a generic Hankel matrix of size $r\times s$. Then $I_t(H)$ is a Knutson ideal for every $t \leq \min (r,s)$.
\end{mthm}

 In particular, this implies that the determinantal ring of a generic Hankel matrix is $F$-pure (see Corollary \ref{corhank}), a result recently proved by different methods in \cite{CMSV}.\par
Furthermore, we characterize all the ideals belonging to the family for this choice of $f$ (see Theorem \ref{thm: car}).\\

\textbf{Acknowledgements.} I am deeply grateful to my advisor, Matteo Varbaro, for suggesting me this problem and for several helpful discussions.  I would also like to thank the anonymous referee for valuable suggestions.

\section{Knutson ideals in any characteristic}
The aim of this section is to try to generalize Knutson's results first to fields of characteristic $p>0$ (not necessarily finite) and then to fields of characteristic $0$.\\

\subsection{Fields of characteristic $p>0$}
Let $\mathbb{K}$ be a field of characteristic $p>0$ and let $f \in S=\mathbb{K}[x_1,\ldots,x_n]$ be a polynomial such that $\lt (f)$ is squarefree for some term order $\prec$. As in the case of $\mathbb{K}=\mathbb{Z}/p\mathbb{Z}$, we can construct the family $\mathcal{C}_f$ as the smallest set of ideals such that:
\begin{itemize}
\item[•] $(f) \in \mathcal{C}_f$
\item[•] $I \in \mathcal{C}_f, \quad J \subseteq S \Rightarrow I:J \in \mathcal{C}_f$
\item[•] $I, J \in \mathcal{C}_f \Rightarrow I+J, \quad I \cap J \in \mathcal{C}_f$.
\end{itemize}
We want to prove the following result.
\begin{prop}\label{pp1} Let $\mathbb{K}$ be  a field of characteristic $p>0$ and let $f$ be a polynomial in $ S=\mathbb{K}[x_1,\ldots,x_n]$ such that $\lt (f)$ is squarefree for some term order $\prec$.
If $I \in \mathcal{C}_f$ then $\lt(I)$ is squarefree.
\end{prop}

A first step toward this result is given by the following observation.

\begin{oss}\label{perfect}
If $\mathbb{K}$ is a perfect field of characteristic $p>0$ then it is easy to generalize Theorem \ref{kn-fp}.\ref{kn-fp1}. In fact \cite[Theorem 4]{Kn} is a consequence of \cite[Theorem 2]{Kn} and the proof of this latter theorem relies on two lemmas, namely \cite[Lemma 2]{Kn} and \cite[Lemma 5]{Kn}. We observe that Knutson gives a proof of \cite[Lemma 2]{Kn} in the case $\mathbb{K}=\mathbb{F}_p$ which easily extends to perfect field of characteristic $p$; the proof is actually the same, one has just to keep in mind that every element $c$ of a perfect field $\mathbb{K}$ of characteristic $p$ has a $p$th root in $\mathbb{K}$ (in the case $\mathbb{K}=\mathbb{F}_p$, $c^p=c$). Furthermore, since \cite[Lemma 5]{Kn} holds for perfect fields of characteristic $p>0$, we get that \cite[Theorem 2]{Kn}, and thus \cite[Theorem 4]{Kn}, extends to polynomial rings over perfect fields of positive characteristic. \par
\end{oss}

To prove Proposition \ref{pp1}, we reduce to the case of perfect fields of positive characteristic so that we can apply Remark \ref{perfect}.\par
In the proof we will need these well known facts.

\begin{prop}(see e.g. \cite[p.46]{Ma})\label{Fact 1}The extension of polynomial rings 
$$S=\mathbb{K}[x_1,\ldots,x_n] \hookrightarrow \overline{S}=\overline{\mathbb{K}}[x_1, \ldots,x_n]$$
is a flat extension. 
\end{prop}

\begin{prop}(see e.g. \cite[Theorem 7.4]{Ma})\label{Fact 2} Let $\pi: A \longrightarrow B$ be a flat ring extension and $I$ and $J$ two ideals of $A$. Then:
\begin{itemize}
\item[\textbf{(i)}] $(I \cap J)B=IB \cap JB$
\item[\textbf{(ii)}]If $J$ is finitely generated then $(I:J)B=IB:JB$.
\end{itemize}
\end{prop} 

\begin{proof}[Proof of Proposition \ref{pp1}]Let $\mathbb{K} \hookrightarrow \overline{\mathbb{K}}$ be the extension of $\mathbb{K}$ to its algebraic closure $\overline{\mathbb{K}}$. Since $\car (\mathbb{K})=p$, then $\overline{\mathbb{K}}$ is a perfect field of characteristic $p$.\par
Let $\bar{S}= \overline{\mathbb{K}}[x_1, \ldots, x_n]$ and consider the natural extension:

\begin{equation*}
\iota: S\longrightarrow \overline{S}.
\end{equation*}

So $\overline{f}:= \iota(f)$ is a polynomial in $\overline{S}$ (we regard $f$ as a polynomial with coefficients in $\overline{K}$). Again we can construct the family $\overline{\mathcal{C}_f}:=\mathcal{C}_{\overline{f}}$ in $\overline{S}$.\par
First of all, we claim that if $I \in \mathcal{C}_f$ then $I\overline{S} \in \overline{\mathcal{C}_f}$. Indeed, by Proposition \ref{Fact 1},
$ \iota :S \longrightarrow \overline{S}$ is a flat extension and we can then use Proposition \ref{Fact 2} to get the following equalities:
\begin{itemize}
\item $(I+J)\overline{S}=I \overline{S}+J \overline{S}$ (always true)
\item $(I \cap J) \overline{S}= I \overline{S}\cap J \overline{S}$ (true for flat extensions)
\item $(I:J)\overline{S}=I \overline{S}:J \overline{S}$ (true for flat extensions and $J$ finitely generated).
\end{itemize}

Consider $(f) \in \mathcal{C}_f$, then $(f)\overline{S}=(\overline{f}) \subseteq \overline{S}$ and $(\overline{f}) \in \overline{\mathcal{C}_f}$ by definition.\par
Now let $I,J \in \mathcal{C}_f$ such that $I \overline{S}, J \overline{S} \in \overline{\mathcal{C}_f}$. By definition $I+J, I \cap J \in \mathcal{C}_f$.\\
Using previous identities, we get

\begin{eqnarray*}
(I+J)\overline{S}=\underbrace{I\overline{S}}_{\in \overline{\mathcal{C}_f}}+\underbrace{J \overline{S}}_{\in \overline{\mathcal{C}_f}} \in \overline{\mathcal{C}_f}\\
(I\cap J)\overline{S}=\underbrace{I\overline{S}}_{\in \overline{\mathcal{C}_f}}\cap\underbrace{J \overline{S}}_{\in \overline{\mathcal{C}_f}} \in \overline{\mathcal{C}_f}.
\end{eqnarray*}

Eventually, let's consider $I \in \mathcal{C}_f$ and $J \subseteq S$ an arbitrary ideal. By definition $I:J \in \mathcal{C}_f$. Suppose $I \overline{S} \in \overline{\mathcal{C}_f}$. Since $J$ is finitely generated, then
$$(I:J) \overline{S}=\underbrace{I\overline{S}}_{\in \overline{\mathcal{C}_f}}:\underbrace{J \overline{S}}_{\subseteq \overline{S}} \in \overline{\mathcal{C}_f}.$$

This proves the claim. Using this result, we can easily conclude the proof. In fact, let $I$ be an ideal in $\mathcal{C}_f$. Then $I \overline{S} \in \overline{\mathcal{C}_f}$ and by Remark \ref{perfect} $\lt(I \overline{S})$ is squarefree beacause we are  working in a polynomial ring over a perfect field of characteristic $p>0$ .
But since Buchberger's algorithm is ``stable" under base extensions, we have

$$\lt(I \overline{S})= \lt(I) \overline{S}.$$

So $\lt(I)$ is squarefree.
\end{proof}
\subsection{Fields of characteristic 0}
Let $\mathbb{K}$ be  a field of characteristic $0$ and let $f \in S=\mathbb{K}[x_1,\ldots,x_n]$ be a polynomial such that $\lt (f)$ is squarefree for some term order $\prec$. As in the previous cases, we can construct the family $\mathcal{C}_f$ as the smallest set of ideals such that:
\begin{itemize}
\item[•] $(f) \in \mathcal{C}_f$
\item[•] $I \in \mathcal{C}_f, \quad J \subseteq S \Rightarrow I:J \in \mathcal{C}_f$
\item[•] $I, J \in \mathcal{C}_f \Rightarrow I+J, \quad I \cap J \in \mathcal{C}_f$.
\end{itemize}

We want to prove the analogous of Proposition \ref{pp1} in the case of polynomial rings over fields of characteristic 0.
\begin{prop}\label{car0}
Let $\mathbb{K}$ be  a field of characteristic $0$ and let $f$ be a polynomial in $ S=\mathbb{K}[x_1,\ldots,x_n]$ such that $\lt (f)$ is squarefree for some term order $\prec$.
If $I \in \mathcal{C}_f$ then $\lt(I)$ is squarefree.
\end{prop}
We know that this holds in polynomial rings over fields of characteristic $p>0$. Using Proposition \ref{pp1}, we will show that the same holds if we are working over fields of characteristic 0.

\subsubsection{Reduction modulo $p$ and initial ideals}
Let $\mathbb{K}$ be a field of characteristic 0 and define $S:= \mathbb{K}[x_1, \ldots,x_n]$. Consider an ideal $I \subseteq S$. Since $I$ is finitely generated, it is always possible to construct a finitely generated $\mathbb{Z}$-algebra $A \subset \mathbb{K}$ such that if $I':= I \cap A[x_1, \ldots, x_n]$ then $I' S=I$. To do so it suffices to 
take $A= \mathbb{Z}[\alpha_1, \ldots, \alpha_s]$ where the $\alpha_i $ are the coefficients of  the generators of $I$ which are not integers.

\begin{es}
If $I=(\sqrt{2} x - \pi y) \subset \mathbb{R}[x,y]$, we can take $A= \mathbb{Z}[\sqrt{2}, \pi]$.
\end{es}

Let $p$ be a prime number which is not invertible in $A$ and fix $P \in \Min(pA)$. The quotient ring $A/P$ is an integral domain of characteristic $p>0$ and we can define $I'_p$ to be the image of $I'$ under the projection map
\begin{equation*}
 \pi \colon A[x_1, \ldots, x_n] \longrightarrow \faktor{A}{P}[x_1, \ldots, x_n].
\end{equation*}

 Since $A/P$ is a domain we can construct its fraction field $F\left(\faktor{A}{P}\right)$ and we define $S_p:=F\left(\faktor{A}{P}\right)[x_1,\ldots,x_n]$. So we can consider the extended ideal $I(p):=I'_p S_p$ in the polynomial ring $S_p$. \par
This is what we call a \emph{reduction modulo $p \in \mathbb{N}$}. Although the notation might be confusing, the ideal $I(p)$ does not depend only on $p$ and $I$ but also on the choice of $P \in \Min (pA)$. \par
To summarize, we have constructed the following diagram:\\

 \begin{center}
 \begin{tikzcd} 
A[x_1, \ldots, x_n] \arrow{r}{} \arrow{d}{\pi} & S=\mathbb{K}[x_1, \ldots, x_n]& \car=0\\ 
\left(\faktor{A}{P}\right)[x_1, \ldots, x_n] \arrow{r}{} & F\left(\faktor{A}{P}\right)[x_1, \ldots, x_n] & \car=p>0
\end{tikzcd}
 \end{center}

Note that the lower map in the diagram is flat.\\

The next lemma states that taking initial ideals commutes with reduction modulo $p$ for all sufficently large $p$.

\begin{lem}\label{redp}
Let $\mathbb{K}$ be a field of characteristic 0 and $S= \mathbb{K}[x_1, \ldots,x_n]$ the polynomial ring over $\mathbb{K}$ with a fixed term order $\prec$. Take $I_1, \ldots,I_m$ ideals in $S$. Then for all $ p \gg 0$ there exists a reduction modulo $p$ such that 
$$\lt(I_j(p))=\lt(I_j)(p) \qquad \forall j=1, \ldots, m.$$
\end{lem}

\begin{proof}
It suffices to prove the result for $m=1$. If $m>1$ we can always choose $p$ greater than the maximum of the $p_j$ such that the result is true for $I_j$ and we are done.\par
Consider an ideal $I=(f_1, \ldots, f_r) \subseteq S$ and
construct the finitely generated $\mathbb{Z}$-algebra $A = \mathbb{Z}[\alpha_1, \ldots, \alpha_t]\subset \mathbb{K}$ where the $\alpha_i $ are the coefficients of  the generators of $I$ which are not integers. Note that since $A$ is a finitely generated $\mathbb{Z}$-algebra, all the polynomial rings we are dealing with are Noetherian.\par
Accordingly with the previous notation, we set 

\begin{equation*}
\begin{split}
I'&:= I \cap A[x_1, \ldots, x_n]\\
I''&:= I'F(A)[x_1, \ldots, x_n]=I \cap F(A)[x_1, \ldots, x_n].
\end{split}
\end{equation*}

Using Buchberger's algorithm we can compute a Gr\"obner basis for $I''$. Let $G_{I''}=\{g_1, \ldots,g_s\}$ be this Gr\"obner basis. Since Buchberger's algorithm is ``stable" under base extensions we get $G_{I''}=G_{I}$, that is $G_{I''}$ is also a Gr\"obner basis for $I$ in $S$.\par
Observe that, possibly multiplying by an element of $A$, we can assume that $g_1, \ldots, g_s$ are polynomials in $A[x_1,\ldots,x_n]$. \par
In the computation of the Gr\"obner basis, no new coefficients appear but we need to invert some elements  $\lambda_1,\ldots,\lambda_t \in A$ to compute S-polynomials. 
If we find a prime number $p$ and a minimal prime $P \in \Min(pA)$ such that $\lambda_1,\ldots,\lambda_t \notin P$, then $\lambda_1,\ldots,\lambda_t $ are invertible in $F\left(\faktor{A}{P}\right)$, so the algorithm is exactly the same also when we reduce modulo $p$. This will imply that $\overline{G_I}=\{\overline{g}_1,\ldots,\overline{g}_s\}$ is a Gr\"obner basis for $I(p)$, hence $$\lt(I(p))=(\lt(\overline{g}_1), \ldots, \lt(\overline{g}_s)).$$\par
Since we are working in Noetherian domains, the principal ideal $(pA)$ has finitely many minimal primes and by Krulls Hauptidealsatz if $P \in \Min(pA)$ then $\hgt(P)=1$. Moreover it's easy to see that if $p$ and $q$ are two different prime numbers, then $\Min(pA) \cap \Min(qA)= \emptyset$. Assume that there exists a prime ideal $Q \in\Min(pA) \cap \Min(qA)$, then $Q \supseteq (pA),(qA)$. In particular $p,q \in Q$ and they are coprime. This would imply that $1\in Q$, a contradiction. Similarly the ideal $(\lambda_i A)$ has finitely many minimal primes of height $1$, therefore there exists a prime number $\overline{p_i}$ such that $$\forall p> \overline{p_i} ,\ \forall P \in \Min (pA) \quad
\lambda_i \notin P  .$$
Taking $\overline{p}:= \max \overline{p_i}$, we get that 
$$\forall p> \overline{p} ,\ \forall P \in \Min (pA) \quad
\lambda_1, \ldots, \lambda_t \notin P .$$
This proves that $\lt(I(p))=(\lt(\overline{g}_1), \ldots, \lt(\overline{g}_s))$ for $p> \overline{p}$. \par 
Using a similar argument we can prove that there exists a prime number $\tilde{p} >0$ such that $$\lt (I) (p)=(\overline{\lt(g_1)},\ldots,\overline{\lt(g_s)})=(\lt(\overline{g}_1), \ldots, \lt(\overline{g}_s)) \qquad \forall p>\tilde{p}.$$\par 
So we can conclude that $\lt (I) (p)=\lt (I (p))$ for $p\gg 0$.
\end{proof}

\subsubsection{Knutson ideals in characteristic 0}

We want to prove Proposition \ref{car0} in characteristic 0. To do so, we reduce to the case of fields of positive characteristic using previous results.\par
As in the case of fields of positive characteristic, we first need to show that if $I \in \mathcal{C}_f$ then $ I(p) \in \mathcal{C}_f (p):=\mathcal{C}_{f(p)}$ for all prime numbers large enough.\\

The following result simplifies our proof, allowing us to prove the result for a single ideal at time using \ref{redp}.



\begin{lem}\label{puguale} TFAE:
\begin{itemize}
\item[1.] $\exists \tilde{p} \gg0$ s.t. $I \in \mathcal{C}_f\Rightarrow I(p) \in \mathcal{C}_f (p) \quad \forall p\geq\tilde{p}$.
\item[2.] $\forall I \in \mathcal{C}_f \quad \exists \tilde{p}_I \gg 0$ s.t. $I(p) \in \mathcal{C}_f (p)\quad \forall p\geq\tilde{p}_I$.
\end{itemize}
\begin{proof}
$1.\Rightarrow 2.$ Obvious.\\
$2.\Rightarrow 1.$ If 2 holds, then $\lt (I(p))$ is squarefree $\forall I \in \mathcal{C}_f$ and for $p \geq\tilde{p}_I$. But we know from Lemma \ref{redp} that $\lt (I(p))=\lt (I)(p)$ for $p$ large enough, so $\lt (I)$ is squarefree for every $I \in \mathcal{C}_f$. By Remark \ref{finite}, we get that $\mathcal{C}_f$ is finite. Once we know that $\mathcal{C}_f$ is finite, we can take $\tilde{p}= \max p_I$ and we are done.

\end{proof}
\end{lem}

\begin{proof}[Proof of Proposition \ref{car0}]
We begin by proving that if $I \in \mathcal{C}_f$ then there exists a prime number $\tilde{p}_I$ such that  $I(p) \in \mathcal{C}_f (p)=\mathcal{C}_{f(p)}$ for all $p \geq \tilde{p}_I$. By Lemma \ref{puguale}, this is equivalent to prove that there exists a prime number $\tilde{p}$ which does not depend on the choice of the ideal, such that if $I\in \mathcal{C}_f$ then $I(p) \in \mathcal{C}_f (p)=\mathcal{C}_{f(p)}$ for all $p \geq \tilde{p}$.\par
Consider $(f) \in \mathcal{C}_f$, then $(f)(p)=(f(p)) \subseteq  S_p$ and as we already explained $(f(p)) \in \mathcal{C}_f (p)=\mathcal{C}_{f(p)}$ for all $p\gg 0$.\par
Now let $I,J \in \mathcal{C}_f$ such that $I(p), J (p) \in \mathcal{C}_f (p)$. By definition of $\mathcal{C}_f$, $I+J, I \cap J \in \mathcal{C}_f$ and we need to prove that $(I+J)(p), (I \cap J)(p) \in \mathcal{C}_f (p)$.\\
Obviously 
\begin{eqnarray*}
(I+J)(p)=\underbrace{I(p)}_{\in \mathcal{C}_f (p)}+\underbrace{J (p)}_{\in \mathcal{C}_f (p)} \in \mathcal{C}_f (p).
\end{eqnarray*}

Now consider the intersection ideal $I \cap J \in S$. It is clear that  $(I \cap J)(p) \subseteq I(p) \cap J(p)$. If we show that they have the same initial ideal, we get 
\begin{eqnarray*}
(I\cap J)(p)=\underbrace{I(p)}_{\in \mathcal{C}_f (p)}\cap\underbrace{J (p)}_{\in \mathcal{C}_f (p)} \in \mathcal{C}_f (p).
\end{eqnarray*}
Using elimination theory and Buchberger's algorithm, we can compute a Gr\"obner basis of $I \cap J$. In fact it is a well know fact (see e.g. \cite[Theorem 11,  p.187]{CLO}) that
$$I \cap J= (tI+(1-t)J) \cap S$$
where $t$ is a new variable and $(tI+(1-t)J)$ is an ideal in $S[t]$ that we are contracting back to $S$. \par

In other words, a Gr\"obner basis of $I \cap J$ is obtained from a Gr\"obner basis of $tI+(1-t)J$ by dropping the elements of the basis that contain the variable $t$ (the so called first elimination ideal with respect to a suitable term order). \par
Therefore
 $$(I \cap J)(p)= ((tI+(1-t)J) \cap S)(p)$$
 $$I(p) \cap J(p)= (tI(p)+(1-t)J(p)) \cap S_p.$$

By Lemma \ref{redp}, $\lt (tI+(1-t)J)(p)=\lt (tI(p)+(1-t)J(p)) $ for all $ p\gg 0$ and we can conclude that  
$$\lt (I \cap J)(p)= \lt (I(p) \cap J(p)).$$

A similar argument works for $I:J$ with $I \in \mathcal{C}_f (p)$ and $J=(f_1,\ldots,f_l) \ \subset S$. In fact it is known (see e.g. \cite[Theorem 11, p.196]{CLO}) that 
$$I:J=\left( \dfrac{1}{f_1}\left( I \cap (f_1)\right)\right) \cap \left(\dfrac{1}{f_2}\left( I \cap (f_2)\right) \right) \cap \ldots \cap \left(\dfrac{1}{f_l}\left( I \cap (f_l)\right) \right). $$

Thus, we can use again elimination theory to compute these intersections and arguing as we have done before, we get that 
\begin{eqnarray*}
(I:J)(p)=\underbrace{I(p)}_{\in \mathcal{C}_f (p)}:\underbrace{J (p)}_{\subseteq S_p} \in \mathcal{C}_f (p).
\end{eqnarray*}
In conclusion, we have proved that if $I \in \mathcal{C}_f$ then $I(p) \in \mathcal{C}_f (p)=\mathcal{C}_{f(p)}$ for all $p$ large enough.\par
Now let $I \in \mathcal{C}_f$. Then $I(p) \in \mathcal{C}_f(p)$  for $p \gg 0$ and by Proposition \ref{pp1} $\lt(I (p))$ is squarefree beacause we are  working in a polynomial ring over a field of positive characteristic.
But we know from Lemma \ref{redp} that 

$$\lt(I(p))=\lt(I)(p) \qquad \forall p \gg 0.$$

So $\lt(I)$ is squarefree.

\end{proof}
 \section{Determinantal ideals of Hankel matrices}
Denote by $X_m^{(l,n)}$ the generic Hankel matrix with $m$ rows and entries $x_l,\ldots,x_n$, that is

\[ X_m^{(l,n)}=
\begin{bmatrix}
    x_{l}       & x_{l+1} & x_{l+2} & \dots & x_{n-m+1} \\
    x_{l+1}       & x_{l+2} & x_{l+3} & \dots & x_{n-m+2} \\
    x_{l+2}       & x_{l+3} & x_{l+4} & \dots & x_{n-m+3} \\
   \vdots & \vdots & \vdots & \ddots & \vdots \\
    x_{l+m-1}       & x_{l+m} & x_{l+m-1} & \dots & x_{n}
\end{bmatrix}.
\]

Note that once we have fixed $m$, $l$ and $n$, the number of columns of  $X_m^{(l,n)}$ is $n-m-l+2$. \par
In particular we are interested in square Hankel matrices of size $m$ and rectangular Hankel matrices of size $m \times (m+1)$. In these cases, if we fix $m$ then $n$ is uniquely determined.  \par
Assume for simplicity that $l=1$:

\begin{minipage}{0.4\columnwidth}

    \[ X_m^{(1,n)}=
\begin{bmatrix}
    x_{1}       & x_{2} & x_{3} & \dots & x_{m} \\
    x_{2}       & x_{3} & x_{4} & \dots & x_{m+1} \\
    x_{3}       & x_{4} & x_{5} & \dots & x_{m+2} \\
   \vdots & \vdots & \vdots & \ddots & \vdots \\
    x_{m}       & x_{m+1} & x_{m+2} & \dots & x_{n}
\end{bmatrix} , 
\]
\end{minipage}
\hfill
\begin{minipage}[c]{0.3 \columnwidth}
\centering
square Hankel matrix:\\ $n=2m-1$
\end{minipage}

\begin{minipage}{0.4\columnwidth}

    \[ X_m^{(1,n)}=
\begin{bmatrix}
    x_{1}       & x_{2} & x_{3} & \dots & x_m& x_{m+1} \\
    x_{2}       & x_{3} & x_{4} & \dots & x_{m+1}& x_{m+2} \\
    x_{3}       & x_{4} & x_{5} & \dots & x_{m+2} & x_{m+3} \\
   \vdots & \vdots & \vdots & \ddots & &\vdots \\
    x_{m}       & x_{m+1} & x_{m+2} & \dots &x_{n-1} & x_{n}
\end{bmatrix},  
\]
\end{minipage}
\hfill
\vspace{5pt}
\begin{minipage}[c]{0.3 \columnwidth}
\centering
Hankel matrix of size $m \times (m+1)$: \\$n=2m.$
\end{minipage}\\

Let $X=X_{m}^{(1,n)}$ be a Hankel matrix and let $t \leq \min(m,n-m+1)$, we denote by $I_t(X)$ the determinantal ideal in $\mathbb{K}[x_1, \ldots,x_n]$ generated by all the $t$-minors of $X$.\\
\begin{oss}\label{osshank}
For $t \leq m\leq n+1-t$ it is known (cf. \cite[Corollary 2.2]{C}) that
 $$I_t(X_{m}^{(1,n)})=I_t(X_{t}^{(1,n)}).$$ That is $I_t(X_{m}^{(1,n)})$ does not depend on $m$ but only on $t$ and $n$.\\

\end{oss}

We now prove that determinantal ideals of a generic square Hankel matrix are Knutson ideals for a suitable choice of $f$.

\begin{thm} \label{prophank}Let $X=X_m^{(1,n)}$ be the square Hankel matrix of size $m$ with entries $x_1, \ldots, x_n$, where $n=2m-1$ and let $f$ be the polynomial $f= \det X \cdot \det X_{m-1}^{(2,n-1)}$  in $S=\mathbb{K}[x_1,\ldots,x_n]$. Then $I_t(X) \in \mathcal{C}_f$ for $t=1,\ldots, m$.

\end{thm}

\begin{proof}
Fix a  diagonal term order $\prec$ on $S$ (that is a monomial term order such that the initial term of each minor is given by the product of its diagonal terms). Then

\begin{equation*}
\begin{split}
\lt (f)&=\lt( \det X) \cdot \lt (\det X_{m-1}^{(2,n-1)})=\\
&= ( x_1 \cdot x_3 \cdots x_n)\cdot ( x_2 \cdot x_4 \cdots x_{n-1})= \prod \limits_{i=1}^{n} x_i.
\end{split}
\end{equation*}

Hence $\lt (f)$ is squarefree and we can construct the Knutson family of ideals associated to $f$.\par
For simplicity of notation, we define

\begin{align*}
&P_1:= X_{m-1 }^{(1,n-1)}&&\text{: rectangular matrix obtained by dropping the last row of X}\\
&P_2:= X_{m }^{(2,n)}&&\text{: rectangular matrix obtained by dropping the first column of X}\\
&Q:= X_{m-1}^{(2,n-1)}&&\text{: square matrix obtained by dropping the last row  and the first}\\
& && \text{  column of X.}
\end{align*}

By Definition \ref{K.I.}, $(f) \in \mathcal{C}_f$ and  $(f): J \in \mathcal{C}_f$ for every ideal $J \subseteq S$. 
Choosing $J= (\det X )$ and $J=(\det Q)$, we get
 
\begin{align} \label{eq1}
\begin{split}
(f): (\det X)&=(\det Q) \in \mathcal{C}_f \\
(f):(\det Q)&=(\det X) \in \mathcal{C}_f.
\end{split}
\end{align}

In particular, $I_m (X)=(\det X) \in \mathcal{C}_f$. This proves the theorem in the case $t=m$.\par
Now let $t=m-1$. It is known (e.g. see \cite{C} and \cite{BC}) that every determinanatal ideal of a generic Hankel matrix $H$ is prime and its height is given by the following formula:

\begin{equation}\label{ht}
\hgt(I_s (H))=n-2s+2
\end{equation} 

where $n$ is the number of variables. In this case
$$\hgt (I_t(X))=2m-1-2(m-1)+2=3.$$
From equalities (\ref{eq1}), taking the sum, we get
$$ I_m(X)+I_{m-1}(Q)=(\det X,\det Q) \in \mathcal{C}_f.$$
Moreover $$ \lt (I_m(X)+I_{m-1}(Q))=( x_1 x_3 \cdots x_n, x_2 x_4 \cdots x_{n-1})$$
is a complete intersection of height 2, so $I_m(X)+I_{m-1}(Q)$ is a complete intersection of height 2 as well.\par
Now observe that $$\hgt(I_t (P_1))= (n-1)-2t+2=2=(2m-1-1)-2(m-1)+2=\hgt(I_t(P_2))$$ and $$I_t (P_1), I_t(P_2) \supseteq (\det X, \det Q)=I_{t+1} (X) + I_{t} (Q) \in \mathcal{C}_f.$$

This means that $I_t (P_1)$ and $I_t(P_2)$ must be minimal primes over the ideal $(\det X, \det Q)\in \mathcal{C}_f$. Thus, they and their sum must be in $\mathcal{C}_f$ by definition.\\
Hence $I_t(X)$ is a prime ideal of height $3$ and it contains the sum of two distinct prime ideals of height 2, namely $I_t (P_1)+I_t(P_2)$. This shows that $I_t(X) \in \mathcal{C}_f$, since it is a minimal prime over $I_t (P_1)+I_t(P_2)$ which is in $\mathcal{C}_f$.\\
In the same way, using a backward inductive argument, it can be proved that $I_t(X) \in \mathcal{C}_f$ for every $t=1, \ldots, m$. Although similar to the previous case, we think it is worth to include the proof of the general case; this might help the reader to understand better the scheme of the proof, which will be discussed later on in the paper. \par
Assume that $I_t(X), I_t(P_1), I_t(P_2), I_t(Q) \in \mathcal{C}_f$; we want to prove that the same holds for $t-1$. \par
By (\ref{ht}), we know that
$$\hgt (I_{t-1}(Q))=n-2-2(t-1)+2=n-2t+2.$$
and $$\hgt (I_{t}(P_1))= \hgt (I_t(P_2))=n-1-2t+2=n-2t+1.$$
Moreover $$I_{t-1}(Q) \supseteq I_t(P_1)+I_t(P_2)$$ 
and $I_t(P_1)+I_t(P_2) \in \mathcal{C}_f$ by induction. So $I_{t-1}(Q) $ must be minimal over $I_t(P_1)+I_t(P_2)$. This proves that $I_{t-1}(Q)\in \mathcal{C}_f $.\par
As a consequence we get that $I_t(X)+I_{t-1}(Q) \in \mathcal{C}_f$. This ideal is the sum of two distinct prime ideals of height $n-2t+2$ and it is contained in $I_{t-1}(P_1)$ and $I_{t-1}(P_2)$ which are two prime ideals of height one more, that is $n-2t+3$. Hence we have that $I_{t-1}(P_1)$ and $I_{t-1}(P_2)$ are miniaml primes over the sum $I_t(X)+I_{t-1}(Q)$ which is in $\mathcal{C}_f$ and so they must be in $\mathcal{C}_f$.
It remains to show that $I_{t-1} (X) \in \mathcal{C}_f$. To do so, one can observe that 
$$I_{t-1}(X) \supseteq I_{t-1}(P_1)+I_{t-1}(P_2) \in \mathcal{C}_f.$$

Hence $I_{t-1}(X)$ is a prime ideal of height $n-2t+4$ that contains the sum of two distinct prime ideals in $\mathcal{C}_f$ of height $n-2t+3$. Thus $I_{t-1}(X)$ must be a minimal prime over 
$I_{t-1}(P_1)+I_{t-1}(P_2) \in \mathcal{C}_f$. By definition, we get $I_{t-1}(X) \in \mathcal{C}_f$. This completes the proof.

\end{proof}

A similar result holds for Hankel matrices of size $m \times (m+1)$.

\begin{thm}\label{prophank2} Let $X=X_{m}^{(1,n)}$ be the rectangular Hankel matrix of size $m \times (m+1)$ with entries $x_1, \ldots, x_n$, where $n=2m$ and let $f$ be the polynomial  $f= \det X_{m}^{(1,n-1)} \cdot \det X_{m}^{(2,n)}$  in $S=\mathbb{K}[x_1,\ldots,x_n]$. Then $I_t(X) \in \mathcal{C}_f$ for $t=1,\ldots, m$.

\end{thm}
\begin{proof}
In this case we define

\begin{align*}
&P_1= X_{m }^{(1,n-1)}&& \text{: square matrix obtained by dropping the last column of X}\\
&P_2= X_{m}^{(2,n)} && \text{: square matrix obtained by dropping the first column of X}\\
&Q= X_{m }^{(2,n-1)}&&\text{: rectangular matrix obtained by dropping the first and the }\\
& && \quad \text{last column of X.}
\end{align*}

Then the proof is similar to that of the case of square Hankel matrices.
\end{proof}

From the previous theorems, we can derive an alternative proof of \cite[Theorem 4.1]{CMSV}.
\begin{cor} \label{corhank} Let $H$ be a generic Hankel matrix of size $r\times s$. Then \\
\emph{(a)} $I_t(H)$ is a Knutson ideal for every $t \leq \min (r,s)$.\\
\emph{(b)} If $\mathbb{K}$ is a field of positive characteristic, then $S/I_t(H)$ is F-pure.
\end{cor}

\begin{proof}
(a) Using Remark \ref{osshank}, we may assume that the Hankel matrix  $H$ has the right size (that is $m {\times}m$ or  $m{\times} (m+1)$), so we can apply  Theorem \ref{prophank} or Theorem \ref{prophank2}. \par
(b) We may assume that $\mathbb{K}$ is a perfect field of positive characteristic. In fact, we can always reduce to this case by tensoring with the algebraic closure of $\mathbb{K}$ and the $F$-purity property descends to the non-perfect case. Using Lemma 4 in \cite{Kn}, we know that the ideal $(f)$ is compatibly split with respect to the Frobenius splitting defined by $\Tr (f^{p-1}\bullet)$ (where $f$ is taken to be as in the previous theorems). Thus all the ideals belonging to $\mathcal{C}_f$ are compatibly split with respect to the same splitting, in particular $I_t(H)$. This implies that such Frobenius splitting of $S$ provides a Frobenius splitting of $S/I_t (H)$. Being $S/I_t(H)$ $F$-split, it must be also $F$-pure.
\end{proof}

Proving Theorem \ref{prophank}, it comes out that determinantal ideals of certain submatrices of Hankel matrices are Knutson ideals. Since we know that $\mathcal{C}_f$ is finite, it is natural to ask whether they are all the ideals belonging to the family or not.\\

The only way to construct new ideals in $\mathcal{C}_f$ starting from two ideals belonging to the family is taking their sums, their intersections and their minimal primes. So we have to control that in the algorithm we used to prove Theorem \ref{prophank} we take all possible sums, intersections and minimal primes of ideals in $\mathcal{C}_f$.\par
The previous algorithm proceeds according to the scheme below:

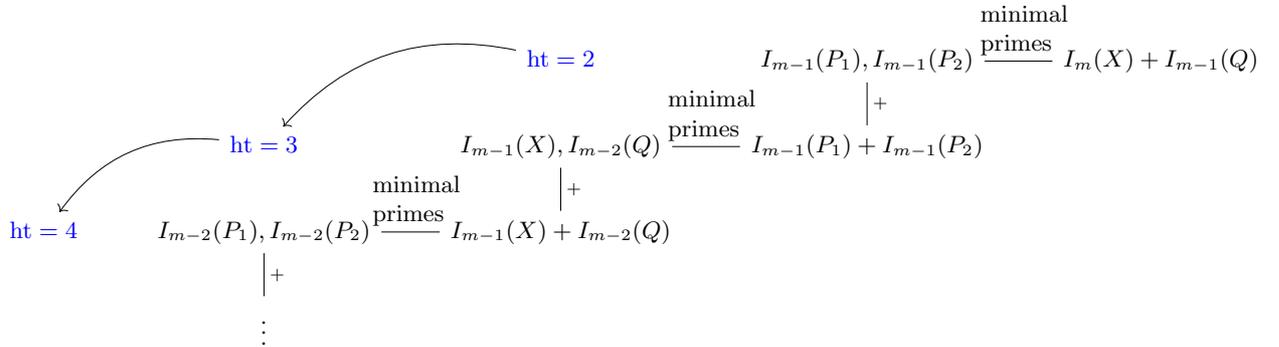
\scriptsize
 \begin{center}
 \begin{tikzcd} 
&& \color{blue}{\hgt =2}\arrow[bend right=30,swap]{ld}{}& I_{m-1}(P_1),I_{m-1}(P_2)\arrow[dash,r, "\scriptsize\parbox{1cm}{minimal primes}"] \arrow[dash]{d}{+}  & I_m(X)+ I_{m-1}(Q)\\ 
 &\color{blue}{\hgt =3}\arrow[bend right=30,swap]{ld}& I_{m-1}(X),I_{m-2}(Q)\arrow[dash,r, "\scriptsize\parbox{1cm}{minimal primes}"]\arrow[dash]{d}{+}& I_{m-1}(P_1)+I_{m-1}(P_2)& \\
 \color{blue}{\hgt =4}& I_{m-2}(P_1),I_{m-2}(P_2)\arrow[dash,r, "\scriptsize\parbox{1cm}{minimal primes}"] \arrow[dash]{d}{+}& I_{m-1}(X)+I_{m-2}(Q)& &\\
  &\vdots  & & &
\end{tikzcd}
 \end{center}

\normalsize
Since two ideals of different height in the scheme are always contained one into the other, if we take their intersection or sum we do not obtain a new ideal. Moreover all the ideals of type $I_t(P_1),I_t(P_2), I_t(X), I_t(Q)$ are prime ideals, so  they are (the only) minimal primes over themselves. If we show that at each step there are no other minimal primes, it turns out that the ideals given by the above procedure are all the possible ideals belonging to the family $\mathcal{C}_f$, that is:\par
\begin{thm}\label{thm: car}
Let $X=X_m^{(1,n)}$ be the square Hankel matrix of size $m$ with entries $x_1, \ldots, x_n$ and let $f= \det X \cdot \det X_{m-1}^{(2,n-1)} \in S=\mathbb{K}[x_1,\ldots,x_n]$. Then the only ideals belonging to $\mathcal{C}_f$ are those of the form
$$I_t(P_1),I_t(P_2), I_t(X), I_t(Q),I_t(X)+ I_{t-1}(Q),I_{t-1}(P_1)+I_{t-1}(P_2).$$
\end{thm} 
By the above discussion, to prove Theorem \ref{thm: car}, it is enough to prove the following:\par 
\begin{prop}\label{primaryDec} With the notation introduced before, we get the following primary decompositions:
\item[1.]$I_t(X)+ I_{t-1}(Q)=I_{t-1}(P_1) \cap I_{t-1}(P_2)$
\item[2.] $I_{t-1}(P_1)+I_{t-1}(P_2)=I_{t-1}(X) \cap I_{t-2}(Q).$
  
\end{prop}

The inclusion $\subseteq$ is obvious in both cases. It remains to prove the reverse inclusion. To do so we will apply the following result which is a consequence of \cite[Corollary 4.6.8]{BH}.

\begin{lem}\label{lemmino}
Let $I,J$ be two ideals in a polynomial ring $S$ such that the following conditions hold:
\begin{enumerate}[topsep=0.4pt,itemsep=0.2pt]
\item $\hgt(I)=\hgt(J)=:h$
\item $I \subseteq J$
\item $\hgt (P)=h \quad \forall P \in \Ass(I)$
\end{enumerate}
then
$$ e(S/I)=e(S/J)\Rightarrow I=J.$$
\end{lem}

Furthermore, in the proof of Proposition \ref{primaryDec} we will need to apply recursively a result by Peskine e Szpiro (see Proposition \ref{PS}) to prove that the ideals $I_{t-1}(P_1) +I_{t-1}(P_2)$ and $I_t(X)+I_{t-1}(Q)$ are Gorenstein for every $t=1, \ldots,m$ and that $$\hgt(I_{t-1}(P_1) +I_{t-1}(P_2))=\hgt(I_{t}(X) +I_{t-1}(Q))+1.$$ 
By the purity of Macaulay, this will imply that all the three conditions of Lemma \ref{lemmino} are staisfied.\par
\begin{prop}\cite[Remark 1.4]{PS}\label{PS}
Let $I$ and $J$ be two homogeneous ideals in a polynomial ring $S$ with no associated primes in common and suppose that $S/(I \cap J)$ is Gorenstein. Then:
\begin{enumerate}
\item $S/I$ is Cohen-Macaulay if and only if $S/J$ is Cohen-Macaulay.
\item If $S/I$ is Cohen-Macaulay, then $\faktor{S}{(I+J)}$ is Gorenstein and $$\hgt(I+J)=\hgt(I)+1.$$
\end{enumerate}
\end{prop}

Collecting together all these results, we can prove Proposition \ref{primaryDec}.

\begin{proof}[Proof of Proposition \ref{primaryDec}]
For $k \geq 1$ we define:
\begin{itemize}
\item[-] $I_k:=I_{m-h}(X)$ and $J_k:=I_{m-h-1}(Q)$, if $k=2h+1$
\item[-] $I_k:=I_{m-h}(P_1)$ and $J_k:=I_{m-h}(P_2)$, if $k=2h$.
\end{itemize}

 We want to show that $I_k+J_k=I_{k+1} \cap J_{k+1}$ for every $1\leq k \leq 2(m-2)+1$. We proceed by induction on $k$, following the usual scheme.\par
First of all observe that all these ideals are different homogenous prime ideals of height $k$ in $S$ (in particular, they have no associated prime ideals in common) and since they are determinantal ideals, they are also Cohen-Macaulay.\par
Assume $k=1$. Then $I_1+J_1=I_m(X)+I_{m-1}(Q)=(\det X, \det Q)$ and $I_2 \cap J_2 =I_{m-1}(P_1) \cap I_{m-1}(P_2)$. We know that $I_1+J_1 \subseteq I_2 \cap J_2$ and that $I_1+J_1$ is a complete intersection of height 2. In particular it is Gorenstein and by the purity of Macaualy, all its associated primes $P$ have the same height, namely $\hgt(P)=\hgt(I_1+J_1)=2$. Moreover $\hgt(I_2 \cap J_2)=2=\hgt(I_1+J_1)$. Hence $I_1+J_1$ and $I_2\cap J_2$ satisfy all the hypothesies of Lemma \ref{lemmino}. If we show that they have the same multiplicity, we get the desired equality.\par
Since $I_1+J_1$ is a complete intersection, we have that $e(I_1+J_1)=m(m-1)$. Moreover the $h$-vector of the determinantal ring of a Hankel matrix $H$ of size $t \times s$ is well known. In fact, being $\hgt (I_t(H))=n-2t+2$ and using Remark \ref{osshank}, the Eagon-Northcott complex provides a minimal free resolution of $S/I_t(H)$. In particular $S/I_t(H)$ is Cohen-Macaulay and has linear resolution. Therefore:

\begin{equation} \label{hvec}
h^{S/I_t(H)}=\left(1,(s-t+1), \binom{s-t+2}{2},\cdots, \binom{s-1}{t-1}\right)
\end{equation}

and its multiplicity is 

\begin{equation} \label{mul}
e(S/I_t(H))=1+(s-t+1)+ \binom{s-t+2}{2}+ \cdots+ \binom{s-1}{t-1}.
\end{equation}

Using this formula we get:

\begin{align*}
e(I_2 \cap J_2)&= e(I_{m-1}(P_1)) +e( I_{m-1}(P_2))=2 e(I_{m-1}(P_1))\\
&= 2\left( 1+(m-m+1+1)+ \binom{3}{2}+\binom{4}{3}+ \cdots+ \binom{m-1}{m-2}\right)\\
&=2(1+2+3+4+\cdots+(m-1))\\
&=2\binom{m}{2}=m(m-1).
\end{align*}

Hence $e(I_1+J_1)=e(I_2 \cap J_2)$ and by Lemma \ref{lemmino} we get $I_1+J_1=I_2 \cap J_2$. Furthermore, using Lemma, \ref{PS} we get that $I_2+J_2$ is Gorestein and $\hgt (I_2+J_2)=\hgt(I_{m-1}(P_1))+1=3$.\par
Now assume $k=2$. Then $I_2+J_2=I_{m-1}(P_1) + I_{m-1}(P_2)$ and $I_3 \cap J_3=I_{m-1}(X)\cap I_{m-2}(Q)$. From the previous case, we know that $I_2+J_2$ is Gorenstein and it has height $3$. As a consequence of the purity theorem of Macaulay we have that $\hgt(P)=\hgt(I_2+J_2)$ for all the associated primes $P$ of $I_2+J_2$. In addition we know that $I_2+J_2 \subseteq I_3 \cap J_3$ and that they have the same height. Again $I_2+J_2$ and $I_3 \cap J_3$ satisfy all the hypothesies of Lemma \ref{lemmino}. If we show that they have the same multiplicity, we get the desired equality. \par 
Iterating this procedure, we get the thesis. More generally, let $k\geq 2$. By induction we may assume that $I_k \cap J_k=I_{k-1} +J_{k-1}$ is Gorenstein and that $\hgt(I_k+J_k)=k+1=\hgt (I_{k+1} \cap J_{k+1})$. Since  $I_k + J_k\subseteq I_{k+1} \cap J_{k+1}$, if we show that $e(I_k + J_k)= e( I_{k+1} \cap J_{k+1})$, by lemma \ref{lemmino} we get $I_k + J_k=I_{k+1} \cap J_{k+1}$ and using Lemma \ref{PS}, we obtain that $I_{k+1} + J_{k+1}$ is Gorenstein of height $(k+1)+1$. \par
Therefore it is enough to show that $e(I_k + J_k)= e( I_{k+1} \cap J_{k+1})$ for every $k$. In other words, we need to prove the following equalities:

\begin{enumerate}
\item[•]$e(I_{t}(P_1)+I_{t}(P_2))=e(I_{t}(X)\cap I_{t-1}(Q))$
\item[•]$e(I_t(X)+ I_{t-1}(Q))=e(I_{t-1}(P_1)\cap I_{t-1}(P_2)).$
\end{enumerate}

To compute the multiplicity of these ideals, we first compute their $h$-vectors. Let $I:=I_t(P_1)$ and $J:=I_t(P_2)$ and consider the following exact sequence:
$$0 \longrightarrow S/(I \cap J) \longrightarrow S/I \oplus S/J \longrightarrow S/(I+J)\longrightarrow 0.$$
By additivity of Hilbert series on short exact sequence, we get:

\begin{equation}\label{hs}
HS_{S/(I +J)}(t)=HS_{S/I \oplus S/J}(t)-HS_{S/(I \cap J)}(t).
\end{equation}

From the previous discussion we already know that $$\hgt (I_t(P_1)+I_t(P_2))=h+1$$ where $h:=\hgt (I_t(P_1))=\hgt (I_t(P_2))=\hgt (I_t(P_1)\cap I_t(P_2))$.\par
This implies that 
$$\dim S/(I_t(P_1)\cap I_t(P_2))=\dim S/I_t(P_1)= \dim I_t(P_2)=n-h=:d$$
and
$$\dim S/(I_t(P_1)+ I_t(P_2))=d-1.$$

Using the well known fact that the Hilbert series is a rational function (see e.g. [BH, Corollary 4.1.8]), we get

\begin{equation} \label{hs2}
\frac{h^{S/(I+J)}(z)}{(1-z)^{d-1}}=\frac{h^{S/I \oplus S/J}(z)-h^{S/(I \cap J)}(z)}{(1-z)^d},
\end{equation}

 hence
 
$$h^{S/(I+J)}(z)=\frac{h^{S/I \oplus S/J}(z)-h^{S/(I \cap J)}(z)}{(1-z)}.$$

It is straightforward to see that $S/I$ and $S/J$ have the same $h$-vector, namely:
$$h^{S/I}=h^{S/J}=(h_{0}^{S/I},h_{1}^{S/I}, \ldots, h_{t-1}^{S/I})$$
where $h_i^{S/I}=\binom{2m-2t+i-1}{i}$ for $ i \leq t-1$. As a consequence, we have that $e(I\cap J)=e(I)+e(J)=2e(I)$.\par
Let $\overline{S}/\overline{I\cap J}$ be the Artinian reduction of $S/(I\cap J)$. Since $I$ and $J$ are generated in degree $t$, for $i<t$ we have  $$h_i^{S/(I\cap J)}=h_i^{\overline{S}/\overline{I\cap J}}=\dim \overline{S}_i=\binom{2m-2t+i-1}{i}=h_i^{S/I}.$$Thus $h_i^{S/(I\cap J)}=h_i^{S/I}$ for $i<t$. But $I\cap J$ is a Gorenstein ideal, so its $h$-vector must be symmetric and we already know that $e(I\cap J)=2e(I)$. This implies that 

$$h^{S/I \cap J}=(h_{0}^{S/I},h_{1}^{S/I}, \ldots, h_{t-2}^{S/I}, h_{t-1}^{S/I},h_{t-1}^{S/I},h_{t-2}^{S/I}, \ldots, h_{1}^{S/I}, h_{0}^{S/I}).$$

Substituting in (\ref{hs2}), we get:

\begin{equation*}
\begin{split}
h^{S/(I+J)}(z)&=\frac{2h^{S/I}(z)-h^{S/(I \cap J)}(z)}{(1-z)}=\\
&=\frac{2\sum \limits_{i=0}^{t-1} h_i^{S/I} z^i- \sum \limits_{i=0}^{t-1} h_i^{S/I} z^i-\sum \limits_{i=t}^{2t-1} h_{2t-1-i}^{S/I} z^i}{1-z}=\\
&= \frac{h_{0}^{S/I}+h_{1}^{S/I}z+ \cdots+ h_{t-1}^{S/I}z^{t-1}-h_{t-1}^{S/I}z^{t}- \cdots -h_{1}^{S/I}z^{2t-2}-h_{0}^{S/I}z^{2t-1}}{1-z}.
\end{split}
\end{equation*}

Dividing by $1-z$, we finally obtain 

\begin{equation*}
h^{S/(I+J)}=\left( h_{0}^{S/I},h_{0}^{S/I}+h_{1}^{S/I}, \ldots,\sum \limits_{i=0}^{t-2} h_i^{S/I},\sum \limits_{i=0}^{t-1} h_i^{S/I},\sum \limits_{i=0}^{t-2} h_i^{S/I}, \ldots, h_{0}^{S/I}+h_{1}^{S/I},h_{0}^{S/I} \right).
\end{equation*}

Note that a similar argument shows that if we consider $I=I_t(X)$ and $J=I_{t-1}(Q)$, then
$$h^{S/(I+J)}=\left( h_{0}^{S/I},h_{0}^{S/I}+h_{1}^{S/I}, \ldots,\sum \limits_{i=0}^{t-2} h_i^{S/I},\sum \limits_{i=0}^{t-2} h_i^{S/I}, \ldots, h_{0}^{S/I}+h_{1}^{S/I},h_{0}^{S/I} \right).$$

Now we can compute the multiplicities.\par
In fact, using the relation $\binom {j}{k}
+\binom{j}{k+1}=\binom{j+1}{k+1}$ and the identity (\ref{hvec}), we get the relation

\begin{align*}
h_{i}^{S/I_{t}(P_1)}=h_{i}^{S/I_{t}(P_2)}=h_{i}^{S/I_{t}(X)}-h_{i-1}^{S/I_{t}(X)}.
\end{align*}

From the $h$-vector of $\faktor{S}{(I+J)}$, using the fact that $h_{i}^{S/I_{t}(X)}=h_{i}^{S/I_{t-1}(Q)}$, we have: 

\begin{align*}
e(I_{t}(P_1)+I_{t}(P_2))= & h_{0}^{S/I_{t}(P_1)}+\left( h_{0}^{S/I_{t}(P_1)}+h_{1}^{S/I_{t}(P_1)}\right)+ \cdots\\
\cdots&+\sum \limits_{i=0}^{t-2} h_i^{S/I_{t}(P_1)}+\sum \limits_{i=0}^{t-1} h_i^{S/I_{t}(P_1)}+\sum \limits_{i=0}^{t-2} h_i^{S/I_{t-1}(P_1)}+ \cdots\\
\cdots & +\left( h_{0}^{S/I_{t}(P_1)}+h_{1}^{S/I_{t}(P_1)}\right)+h_{0}^{S/I_{t}(P_1)}=\\
=&\underbrace{h_{0}^{S/I_{t}(X)}+h_{1}^{S/I_{t}(X)}+\cdots+h_{t-2}^{S/I_{t}(X)}+h_{t-1}^{S/I_{t}(X)}}_{=e(I_{t}(X))}+\\
&+ \underbrace{h_{t-2}^{S/I_{t}(X)}+\cdots+h_{1}^{S/I_{t}(X)}+h_{0}^{S/I_{t}(X)}}_{=e(I_{t-1}(Q))}=\\
=&e(I_{t}(X))+e(I_{t-1}(Q))=e(I_{t}(X)\cap I_{t-1}(Q)).
\end{align*}

So the first equality has been proved.\par
For the second equality, one can argue in a similar way observing that
$$h_{i}^{S/I_{t}(X)}=h_{i}^{S/I_{t-1}(P_1)}-h_{i-1}^{S/I_{t-1}(P_1)}.$$
Computing the multiplicity of $I_{t}(X)+ I_{t-1}(Q)$ from its $h$-vector, we get

\begin{align*}
e(I_{t}(X)+ I_{t-1}(Q))=& h_{0}^{S/I_{t}(X)}+\left(h_{1}^{S/I_{t}(X)}+h_{0}^{S/I_{t}(X)}\right)+\cdots+\sum \limits_{i=0}^{t-2} h_i^{S/I_{t}(X)}+\\
+&\sum \limits_{i=0}^{t-2} h_i^{S/I_{t}(X)}+ \cdots +\left( h_{0}^{S/I_{t}(X)}+h_{1}^{S/I_{t}(X)}\right)+h_{0}^{S/I_{t}(X)}=\\
=&2\left(h_{0}^{S/I_{t-1}(P_1)}+h_{1}^{S/I_{t-1}(P_1)}+\cdots+h_{t-2}^{S/I_{t-1}(P_1)}\right)=\\
=&2\left(e(I_{t-1}(P_1)\right)=e(I_{t-1}(P_1) \cap I_{t-1}(P_2)).
\end{align*}
\end{proof}

\begin{oss}
It is worth noticing that, while proving Proposition \ref{primaryDec}, we also found out that
the ideals $I_t(X)+ I_{t-1}(Q)=I_{t-1}(P_1) \cap I_{t-1}(P_2)$ and $I_{t-1}(P_1)+I_{t-1}(P_2)=I_{t-1}(X) \cap I_{t-2}(Q)$ are Gorenstein ideals for every $t$.\par
Furthermore, we have computed the following $h$-vectors:\\

\emph{(a)} Let $I=I_t(P_1)$ and $J=I_t(P_2)$. Then:
  $$h^{S/(I+J)}=\left( h_{0}^{S/I},h_{0}^{S/I}+h_{1}^{S/I}, \ldots,\sum \limits_{i=0}^{t-2} h_i^{S/I},\sum \limits_{i=0}^{t-1} h_i^{S/I},\sum \limits_{i=0}^{t-2} h_i^{S/I}, \ldots, h_{0}^{S/I}+h_{1}^{S/I},h_{0}^{S/I} \right).$$
  
\emph{(b)} Let $I=I_t(X)$ and $J=I_{t-1}(Q)$. Then:
$$h^{S/(I+J)}=\left( h_{0}^{S/I},h_{0}^{S/I}+h_{1}^{S/I}, \ldots,\sum \limits_{i=0}^{t-2} h_i^{S/I},\sum \limits_{i=0}^{t-2} h_i^{S/I}, \ldots, h_{0}^{S/I}+h_{1}^{S/I},h_{0}^{S/I} \right).$$
Note that these $h$-vectors are unimodal, as $h_i^{S/I}$ is non-negative for every $i$. It should be stressed that this is expected by the $g$-conjecture since we have proved that $I+J$ is always Gorenstein.\\
\end{oss}

\begin{oss}
We have shown that if $f= \det X \det Q$ then $S/I$ is Cohen-Macaulay for every ideal $I \in \mathcal{C}_f$. We want to point out that this fact is proper of this specific choice of $f$: if we consider for example $f=x_1 \cdots x_n$ then $\mathcal{C}_f$ is the family of all the squarefree monomial ideals of $S$ and most of them are not Cohen-Macaulay.

\end{oss}

 \end{document}